\documentclass[12pt,a4paper,reqno]{amsart}
\usepackage[hmargin=2.5cm,bmargin=2.5cm,tmargin=3cm]{geometry}
\usepackage{enumerate}
\usepackage{amsmath,amsthm,amssymb,amsfonts,latexsym}
\usepackage{blkarray}
\usepackage[gen]{eurosym}
\usepackage[german,english]{babel}

\usepackage{latexsym}

\usepackage{tikz}					

     \newcommand{\NN}{\mathbb{N}}

     \newcommand{\RR}{\mathbb{R}}

     \newcommand{\cG}{\mathcal{G}}

	\newcommand{\sine}{\sin(\ell_\me \omega)}
	\newcommand{\sinf}{\sin(\ell_\mf \omega)}
	\newcommand{\cose}{\cos(\ell_\me \omega)}
	\newcommand{\cosf}{\cos(\ell_\mf \omega)}

\newtheorem{theorem}{Theorem}
\newtheorem{lemma}[theorem]{Lemma}

\newtheorem{proposition}[theorem]{Proposition}

\newtheorem{remark}[theorem]{Remark}

\newcommand{\be}{\begin{equation}}
\newcommand{\ee}{\end{equation}}
\newcommand{\bea}{\begin{eqnarray*}}
\newcommand{\eea}{\end{eqnarray*}}
\newcommand{\beq}{\begin{eqnarray}}
\newcommand{\eeq}{\end{eqnarray}}

\newcommand{\bigzero}{\mbox{\normalfont\Large\bfseries 0}}
\newcommand{\rvline}{\hspace*{-\arraycolsep}\vline\hspace*{-\arraycolsep}}

\newcommand{\mG}{\mathsf G}
\newcommand{\mV}{\mathsf V}
\newcommand{\mw}{\mathsf w}
\newcommand{\mv}{\mathsf v}
\newcommand{\mE}{\mathsf E}
\newcommand{\me}{\mathsf e}
\newcommand{\mf}{\mathsf f}

\usepackage{multibib}
\newcites{group}{Quoted publications by members of the group}
\newcites{other}{Quoted publications by other authors}

\usepackage{todonotes}

\title{On fully supported eigenfunctions of quantum graphs} 



\author[M.~Pl\"umer]{Marvin Pl\"umer}
\author[M.~T\"aufer]{Matthias T\"aufer}

\address{Marvin Pl\"umer, Lehrgebiet Analysis, Fakult\"at Mathematik und Informatik, Fern\-Universit\"at in Hagen, D-58084 Hagen, Germany}
\email{marvin.pluemer@fernuni-hagen.de}

\address{Matthias T\"aufer, Lehrgebiet Analysis, Fakult\"at Mathematik und Informatik, Fern\-Universit\"at in Hagen, D-58084 Hagen, Germany}
\email{matthias.taeufer@fernuni-hagen.de}

\date{\today}

\thanks{The work of M.P.\ was supported by the Deutsche Forschungsgemeinschaft (Grant 397230547). The authors thank Lior Alons, Gregory Berkolaiko, Joachim Kerner and Delio Mugnolo for helpful comments and discussions. 
}

\begin{document}
\begin{abstract}
	We prove that every metric graph which is a tree has an orthonormal sequence of Laplace-eigenfunctions of full support. 
	This implies that the number of nodal domains $\nu_n$ of the $n$-th eigenfunction of the Laplacian with standard conditions satisfies $\nu_n/n \to 1$ along a subsequence and has previously only been known in special cases such as mutually rationally dependent or rationally independent side lengths.
	It shows in particular that the Pleijel nodal domain asymptotics from two- or higher dimensional domains cannot occur on these graphs: Despite their more complicated topology, they still behave as in the one-dimensional case.
	We prove an analogous result on general metric graphs under the condition that they have at least one Dirichlet vertex.
	Furthermore, we generalize our results to Delta vertex conditions and to edgewise constant potentials.
	The main technical contribution is a new expression for a secular function in which modifications to the graph, to vertex conditions, and to the potential are particularly easy to understand.
\end{abstract}

\maketitle
\section{Introduction}

\emph{Metric graphs} are metric spaces constituted of intervals which are glued together at their endpoints according to the structure of a combinatorial graph.
They can be understood as one possible generalization of intervals, and one can see them as objects between one- and higher-dimensional domains. A \emph{quantum graph} is a self-adjoint operator defined upon a metric graph. One such operator is the Laplace operator (or \emph{Laplacian}), which operates edgewise as the second derivative on those square summable functions that satisfy prescribed vertex conditions.
In this note, we focus on standard, Dirichlet and $\delta$-coupling conditions.
The Laplacian can have some surprising features such as eigenfunctions vanishing on some edges -- a violation of the unique continuation principle and a metric-graph specific phenomenon which never happens in connected domains in any dimension.

One motivation for this article arises from the question whether the $\limsup \nu_n/n$, where $\nu_n$ denotes the number of so-called nodal domains of the $n$-th eigenfunction, always behaves as in the one-dimensional case, that is $\limsup \nu_n/n = 1$. 
In other words, we ask whether the topological complexity of metric graphs alone is indeed unable to warp the nodal domain asymptotics in such a way that they differ from the one-dimensional case and exhibit higher dimensional features.
This has been proved only under technical assumptions so far.
Our first result, Theorem~\ref{thm:main_tree} proves that this is indeed universally true on tree graphs with standard conditions at all vertices, thus partly settling a conjecture asked in~\cite{HofmannKMP-20}.
Theorem~\ref{thm:main_Dirichlet} generalizes this to all connected graphs with at least one Dirichlet vertex.
Theorems~\ref{thm:main_delta} and~\ref{thm:main_potential} show that these results are robust when standard conditions are replaced by $\delta$-coupling conditions and when an edgewise constant potential is added to the Laplacian.

	Our technical novelty is a new expression for a secular function, that is an analytic function the positive zeros of which are in one-to-one correspondence with the spectrum of the Laplacian.
	Our secular function differs from usual expressions which are constructed as $\det (\operatorname{Id} - U(\omega))$ for a unitary matrix $U$.
	The latter function boasts the advantage that one can incorporate all possible self-adjoint boundary conditions in the theory and that the order of its (positive) zeros coincides with the multiplicity of the eigenvalue. 
	Our secular function has the advantage that it is a priori real-valued and that we can explictly understand the effects of some graph operations on it.
	Some topological properties of the graph are also directly observable from the secular function, see Proposition~\ref{prop:ker_A_0}.
	
	
We introduce notation and main results in Section~\ref{sec:intro_and_results}.
The secular function is introduced in Section~\ref{sec:secular_function} where also Theorem~\ref{thm:main_tree} is proved.
Sections~\ref{sec:Dirichlet}, \ref{sec:delta}, and \ref{sec:potential} contain the proofs of Theorems~\ref{thm:main_Dirichlet}, \ref{thm:main_delta}, and \ref{thm:main_potential}.

\section{Preliminaries and main results}
	\label{sec:intro_and_results}

 Let \(\mathcal G\) be a \emph{compact and connected metric graph}, that is a finite, connected combinatorial graph \(\mG\) with edge set \(\mE=\mE_\mathcal G\) and vertex set \(\mV=\mV_\mathcal G\) where each edge \(\me\) is identified with an interval \([0,\ell_\me]\) of finite length \(\ell_\me\in (0,\infty)\). 
 For a rigorous definition of metric graphs as metric measure spaces see also \cite{Mugnolo-19}. 
 We write \(\mathcal G=(\mG,\ell)\) where \(\ell = (\ell_\me)_{\me\in\mE}\in \mathbb R^{|\mE|}\) denotes the vector of edge lengths. 
 The \emph{initial} and \emph{terminal vertex} of \(\me\) is the vertex \(\mv\in\mV\) incident to \(\me\) which is identified with \(0\) or \(\ell_\me\), respectively. 
 This suggests the notation \(\me=\mv\mw\) for an edge $\me$ connecting the initial vertex $\mv$ with the terminal vertex $\mw$ which we will use when appropriate. 
The \emph{degree} of a vertex $\mv$ is the number of edges $\mv$ belongs to. 
 
 We are interested in spectral properties of the Laplacian \(-\Delta=-\Delta_\mathcal G\) on \(\mathcal G\) that acts edgewise as the negative second derivative \(-\frac{\mathrm{d}^2}{\mathrm{d} x_\me^2}\) and whose domain is the space of functions \(f\) that are edgewise in the Sobolev space \(H^2(0,\ell_\me)\) and satisfy for each vertex \(\mv\in\mV\) one of the following vertex conditions:
\begin{itemize}
	\item \emph{Standard} or \emph{Kirchhoff-Neumann conditions}: \(f\) is continuous in \(\mv\) and satisfies the Kirchhoff condition
	\begin{equation}\label{eq:kirchhoff-condition}
		\sum_{\me\in\mE_\mv} \partial_{\me}f(\mv)=0
	\end{equation}
	where \(\mE_\mv\) denotes the set of edges incident to \(\mv\) and \(\partial_{\me}f(\mv)\) is the outward derivative of \(f\) on \(\me\) at \(\mv\);
	\item \emph{Dirichlet conditions}: \(f\) takes the value \(0\) in \(\mv\).
\end{itemize}
We denote by \(\mV_N\) the set of vertices in \(\mV\) with standard conditions and by \(\mV_D=\mV\setminus\mV_N\) the set of vertices with Dirichlet vertex conditions. It is well-known that \(-\Delta\) is a non-negative self-adjoint operator with purely discrete spectrum.
In Theorem~\ref{thm:main_delta} we will also consider
\begin{itemize}	
	\item
	\emph{\(\delta\)-coupling conditions}: $f$ is continuous in $\mv$ and satisfies the relation
	\begin{equation}
	\label{eq:delta_conditions}
	\sum_{\me\in\mE_\mv} \partial_\me \psi(\mv) = \alpha_\mv \psi(\mv)
	\end{equation}
	for real coupling constants $( \alpha_\mv )_{\mv \in \mV}$.
\end{itemize}
 Since adding a dummy vertex with standard conditions on an edge does not change the spectral properties of the operators considered in this article, we assume that the underlying combinatorial graph \(\mG\) does not contain any loops, meaning that the initial and terminal vertex of any edge in \(\mE_\mathcal G\) do not coincide. 
 Indeed, this can be achieved by adding such a vertex on every such loop. 
 Furthermore, we assume that all vertices with Dirichlet conditions are of degree one. 
 Indeed, if a Dirichlet vertex was of higher degree $k > 1$, then we could replace it by $k$ distinct Dirichlet vertices of degree one without affecting the spectrum of the Laplacian.
 We also emphasize that the assumed connectedness of $\cG$ refers to connectedness \emph{after} all Dirichlet vertices have been split and turned into vertices of degree one as just described.

Our starting point is Conjecture 4.3 in \cite{HofmannKMP-20} which states that
for every compact connected metric graph there exists an orthonormal base $(\psi_n)_{n \in \NN}$ of Laplace-eigenfunctions with monotonous eigenvalues such that such that for a subsequence $(n_k)_{k \in \NN}$, all $\psi_{k_n}$ have \emph{full support}, i.e. none of these eigenfunctions vanishes identically on any edge of $\cG$.
Together with Proposition 3.4 in~\cite{HofmannKMP-20}, this would imply for this choice of eigenfunctions
\begin{equation}
	\label{eq:Pleijel}
	\limsup_{n \to \infty} \frac{\nu_n}{n}
	= 
	1
\end{equation}
where $\nu_n$ denotes the number of \emph{nodal domains} of $\psi_n$, that is the number of connected components in which $\psi_n$ is non-zero.
In particular, since one already knows the bound $\frac{\nu_n}{n} \leq 1$~\cite{Courant-1923}, the conjecture would show that latter bound is sharp.

Identity~\eqref{eq:Pleijel} holds on bounded intervals since Sturm's Oscillation Theorem~\cite{Sturm-1836} yields $\nu_n = n$ in this case.
As for higher dimensions, Pleijel's celebrated nodal domain theorem states that in dimension $2$, the corresponding $\limsup$ for the Laplacian is \emph{strictly smaller than $1$}~\cite{Pleijel-56}, with later improvements by Bourgain~\cite{Bourgain-13}.
The validity of the conjecture would thus show that Pleijel's Theorem is a truly higher-dimensional effect and cannot be reproduced on metric graphs which, despite having a topologically more complex structure than intervals, are locally still one-dimensional metric spaces.

There is a remarkable dichotomy today of special cases in which the conjecture is known to be true:
On the one hand, these are graphs with \emph{pairwise rationally dependent} side lengths, cf.~\cite[Theorem 4.2]{HofmannKMP-20}.
The crucial idea is to use their resonance and to construct a fully supported eigenfunction by putting cosine waves on all edges.
This allows for an easy construction of fully supported eigenfunctions, but in general (when the graph has loops, i.e. if it is not a tree), these eigenfunctions will be \emph{degenerate}, since eigenfunctions to the same eigenvalue can also be constructed by arranging sine waves on loops.
As we shall see, this degeneracity of eigenvalues prevents the application of certain techniques.
Furthermore, this construction strongly relies on the rational dependence which will break down under arbitrarily small perturbations.

On the other hand, empirically, \emph{rationally independent side lengths} will generate chaos and favour the emergence of fully supported eigenfunctions.
A relevant concept in this context are \emph{generic eigenfunctions}. 
These are eigenfunctions of the graph Laplacian which have multiplicity one and are non-zero at every vertex. 
In particular, they must have full support.
It is known that for a graph $\mG$, the set of vectors of side lengths $( \ell_\me )_{\me \in \mE}$ which lead to an infinite series of generic eigenfunctions is of the second Baire category~\cite{BerkolaikoL-17}, i.e. "generic eigenfunctions are generic". Moreover, for graphs with rationally independent edge lengths, it is was shown in \cite[Proposition A.1]{AlonBB-18} that there exists an infinite sequence of generic eigenfunctions and, thus, the conjecture also holds in that case.
It should be noted that many works explicitly demand the existence of an infinite family of generic eigenfunctions~\cite{Band-14}.
This dichotomy of methods -- resonance in the case of rationally dependent and chaos for rationally independent side lengths -- does not seem amenable to a unification.

Our first result surmounts this at least on trees and shows that \emph{every} metric tree graph has a sequence of generic eigenfunctions.

\begin{theorem}
\label{thm:main_tree}
	Let $\mathcal G$ be a connected and compact metric tree with standard conditions at all vertices of \(\mathcal G\). Then there exist an infinite sequence of generic eigenvalues, that is a strictly increasing sequence of eigenvalues \((\lambda_k)_{k\in\mathbb N}\) of multiplicity one and a sequence of corresponding eigenfunctions \((\psi_k)_{k\in\mathbb N}\) of \(-\Delta\), so that each \(\psi_k\) does not vanish at the vertices of \(\mathcal G\).
\end{theorem}

The proof of Theorem~\ref{thm:main_tree} relies on the secular function approach given in Section~\ref{sec:secular_function}.
A close look at the proof indicates the modifications necessary to also treat metric graphs (not only trees!) with at least one Dirichlet vertex and leads to our next result.
We comment in Remark~\ref{rem:why_only_trees} on the challenges to generalize Theorem~\ref{thm:main_tree} to general graphs without Dirichlet vertices.  

\begin{theorem}
\label{thm:main_Dirichlet}
	Let $\mathcal G$ be a connected and compact metric graph with at least one Dirichlet vertex.
	Then there exist an increasing sequence of eigenvalues \((\lambda_k)_{k\in\mathbb N}\) and a sequence of corresponding eigenfunctions \((\psi_k)_{k\in\mathbb N}\) of \(-\Delta\) with full support.
\end{theorem}

The proof of Theorem~\ref{thm:main_Dirichlet} is given in Section~\ref{sec:Dirichlet}.
Note that we do \emph{not} claim that the eigenvalues corresponding to the fully supported eigenfunctions in Theorem~\ref{thm:main_Dirichlet} have multiplicity one, even though this seems very plausible to us.
Indeed, we believe that the usual secular function approach as in~\cite{Band-14} would also yield this non-degeneracy of the eigenvalues but we are not going to elaborate the argument here since we focus on presenting our version of a secular function as a supplementary technique where the effects of modifications to the graph, the operator, and to the boundary conditions are more accessible.

In particular, we can obtain some more consequences and generalizations.
The first one are \emph{$\delta$-coupling conditions}:

\begin{theorem}
	\label{thm:main_delta}
	The statements of Theorems~\ref{thm:main_tree} and~\ref{thm:main_Dirichlet} remain valid if any number of standard conditions are replaced by \(\delta\)-coupling conditions as in~\eqref{eq:delta_conditions}.
\end{theorem}

Then, we can also modify the operator by adding a \emph{piecewise constant potential}, that is we choose real parameters $(q_\me)_{\me \in \mE}$ and consider the operator $H$ that acts edgewise as negative second derivative plus potential $- \frac{\mathrm{d}^2}{\mathrm{d} x_\me^2} \psi + q_\me \psi$.

\begin{theorem}
	\label{thm:main_potential}
	The statements of Theorems~\ref{thm:main_tree},~\ref{thm:main_Dirichlet}, and~\ref{thm:main_delta} remain valid if we add any real-valued, edgewise constant potential to the Laplacian.
\end{theorem}

Note that step-function potentials that take distinct values on a finite number of connected components of the metric graph are incorporated within Theorem~\ref{thm:main_potential} after adding additional standard vertices at all points where the potential switches values.
It seems therefore plausible that a perturbation argument might also yield a similar result for potentials which can be sufficiently well approximated by step functions, but we are not going to pursue this idea further in this article.
The proof of Theorem~\ref{thm:main_delta} is in Section~\ref{sec:delta} and the proof of Theorem~\ref{thm:main_potential} in Section~\ref{sec:potential}.

\section{A secular function for pedestrians and its application to trees}
\label{sec:secular_function}

In this section, we reformulate the eigenvalue problem for $- \Delta$ in terms of a frequency-dependent finite dimensional matrix $A_\omega$ which has non-trivial kernel if and only if $\omega^2$ is an eigenvalue of $-\Delta$.
The function $\omega \mapsto \det A_\omega$ is a \emph{secular function}, that is a holomorphic function the non-negative zeros of which are in one-to-one correspondence with $\sigma(-\Delta)$.

Secular functions are a common tool in the spectral theory of metric graphs.
They are usually constructed as a determinant of a matrix on a \emph{complex} $2 \lvert \mE \rvert$-dimensional vector space using a scattering approach at the vertices \cite{KottosS-99}.
This leads to an a priori complex-valued secular function, even though it can be made real-valued by a multiplication with a phase, see~\cite[Remark 2.10]{BerkolaikoK-12-book}, and allows to incorporate all possible choices of self-adjoint vertex conditions in one theory~\cite{BolteE-09}.

Our approach seems to be a bit more hands-on. 
We assume continuity at the vertices, which is not required for all self-adjoint vertex conditions, but which is certainly the case for standard, Dirichlet and $\delta$-coupling conditions.
Our matrices are purely real-valued which would a priori require a real $4 \lvert \mE \rvert$-dimensional vector space (recall that this is the real dimension of a complex $2 \lvert \mE \rvert$-dimensional vector space), but thanks to continuity of the eigenfunctions, this reduces to $\lvert \mV \rvert + 2 \lvert \mE \rvert$ dimensions.

More precisely, to define the secular function, we use the following reductions:
At fixed eigenvalue $\lambda = \omega^2$, $\omega \geq 0$, an eigenfunction $\psi$ is uniquely determined by
\begin{itemize}
	\item
	its $\lvert \mV \rvert$ many values at vertices $\{ \psi(\mv) \}_{\mv \in \mV}$,
	\item
	its $2 \lvert \mE \rvert$ many outward derivatives at the end points of edges
	$\{ \partial_\me \psi(\mv) \}_{\mv \in \mV, \me \in \mE_\mv}$.
\end{itemize}
On any edge \(\me = \mv \mw\), an eigenfunction is a linear combination of $\sin$ and $\cos$ waves with frequency $\omega$.
This implies the following consistency condition between function values and derivatives at the endpoints
\begin{equation*}
	\label{eq:consistency-conditions}
	\begin{pmatrix}
		\psi(\mw)\\
		\partial_\me \psi(\mw)\\
	\end{pmatrix}
	=
	\begin{pmatrix}
		\cos(\ell_\me \omega)   & -\frac{\sin(\ell_\me \omega)}{\omega}\\
		- \omega\sin(\ell_\me \omega) & -\cos(\ell_\me \omega)\\
	\end{pmatrix}
	\begin{pmatrix}
		\psi(\mv)\\
		\partial_\me \psi(\mw)\\
	\end{pmatrix}
\end{equation*}
which, for $\omega > 0$, can be rewritten as
\begin{equation}
	\label{eq:consistency-conditions-rescaled}
	\begin{pmatrix}
		\psi(\mw)\\
		\frac{\partial_\me \psi(\mw)}{\omega}\\
	\end{pmatrix}
	=
	\begin{pmatrix}
		\cos(\ell_\me \omega)   & -\sin(\ell_\me \omega)\\
		- \sin(\ell_\me \omega) & -\cos(\ell_\me \omega)\\
	\end{pmatrix}
	\begin{pmatrix}
		\psi(\mv)\\
		\frac{\partial_\me \psi(\mv)}{\omega}\\
	\end{pmatrix},
\end{equation}
or equivalently
\begin{equation}
	\label{eq:consistency-conditions-rescaled-rearranged}
	\begin{pmatrix}
		1 & -\cos(\ell_\me \omega)   & \sin(\ell_\me \omega) & 0\\
		0 & \sin(\ell_\me \omega) & \cos(\ell_\me \omega) & 1\\
	\end{pmatrix}
	\begin{pmatrix}
		\psi(\mw)\\
		\psi(\mv)\\
		\frac{\partial_\me \psi(\mv)}{\omega}\\
		\frac{\partial_\me \psi(\mw)}{\omega}
	\end{pmatrix}
	=0.
\end{equation}
Identity~\eqref{eq:consistency-conditions-rescaled-rearranged} holds on every edge $\me \in \mE$ which leads to $2 \lvert \mE \rvert$ equations with $(\lvert\mV\rvert+2\lvert \mE\rvert )$ variables. 
We combine them with the \(\lvert\mV\rvert\) many vertex conditions (as of now we focus on standard or Dirchlet conditions and discuss $\delta$-coupling conditions in Section~\ref{sec:delta}) into a $\omega$-dependent \( (\lvert\mV\rvert+2\lvert \mE\rvert ) \times ( \lvert\mV\rvert+2\lvert \mE\rvert ) \)-square matrix
\begin{equation}\label{eq:definition-A-omega}	
	A_\omega (\mathcal G,\mathsf V_D)
	=
	\begin{pmatrix}
	\begin{matrix}
		\dots & 0 & \dots \\
		0 & 1 & 0 \\
		 &  \dots &
	\end{matrix}
	&
	\rvline
	&
		\begin{matrix}
		\dots & 1 & \dots & 1 & \dots \\
		&\dots & 0 & \dots & \\
		&&\dots&&
		\end{matrix}
	\\
	\hline
		\begin{matrix}
		1 & - \cos( \ell_\me \omega)\\
		  & \sin( \ell_\me \omega)\\
		  & \dots\\
		\end{matrix}
	&
	\rvline
	&		
		\begin{matrix}
		\sin( \ell_\me \omega) & \\
		\cos( \ell_\me \omega) & 1 \\
		& \dots\\
		\end{matrix}
	\end{pmatrix}.
\end{equation} 
The top $\lvert \mV \rvert$ rows of the matrix represent the vertex conditions, the bottom $2 \lvert \mE \rvert$ rows the consistency conditions \eqref{eq:consistency-conditions-rescaled-rearranged} across edges.
When appropriate, we omit the dependence on the graph $\mathcal G$ and the Dirichlet vertex set \(\mV_D\) and simply write $A_\omega$. 
We will also write \(A_\omega (\mathcal G)\) if \(\mV_D=\emptyset\) is empty, and \(A_\omega (\mathcal G,\mv)=A_\omega (\mathcal G,\mathsf V_D)\) if \(\mV_D=\{\mv\}\) is a one-element set.

Now, to every eigenfunction $\psi$ of \(-\Delta\) with eigenvalue $\lambda = \omega^2 > 0$, corresponds a uniquely determined vector $x = x(\psi) \in \ker A_\omega\subset \RR^{\lvert \mV \rvert + 2 \lvert \mE \rvert}$ given by
\begin{equation}\label{eq:discretization}
	x
	=
	\begin{pmatrix}
		\{ x(\mv) \}_{\mv \in \mV}
		\\
		\{ x(\mv,\me)\}_{\mv \in \mV, \me \in \mE_\mv}
	\end{pmatrix}
	=
	\begin{pmatrix}
		\{ \psi(v) \}_{\mv \in \mV}
		\\
		\{ \frac{\partial_\me \psi(\mv)}{\omega} \}_{\mv \in \mV, \me \in \mE_\mv}
	\end{pmatrix}
\end{equation}
and it is easy to see that vice versa each non-zero vector \(x\in\ker A_\omega\) gives rise to an eigenfunction \(\psi\). In other words, we have proved:

\begin{lemma}\label{lem:isomorphism}
For \(\omega> 0\) and \(\lambda=\omega^2\) the mapping
	\[\mathrm{Eig}(-\Delta,\lambda)\rightarrow \ker(A_\omega), \quad \psi\mapsto x(\psi)\]
is a well-defined isomorphism of vector spaces.
In particular, $\lambda$ is an eigenvalue of $-\Delta$ if and only if $\det A_\omega = 0$.
\end{lemma}
The next lemma shows that this secular function $\omega \mapsto \det A_\omega$ is well-defined:

\begin{lemma}
The above defined secular function $\omega \mapsto \det A_\omega$ does not change when
\begin{enumerate}[(a)]
\item the orientation of an edge is inverted,
\item an edge is subdivided into two edges by a vertex of degree $2$ with standard conditions on the new vertex.
\end{enumerate}
\end{lemma}

\begin{proof}
	For simplicity, we only consider vertices with standard conditions since Dirichlet conditions lead to very similar calculations.
	We frequently use the fact that simultaneous permutations of lines and columns as well as adding multiples of a row or column to another one will not affect the determinant. 
	Denoting the edge by $\me = \mv \mw$, (a) is shown by demonstrating that the following two determinants are equal:
	\begin{align}
	\label{eq:det_oriented_1}
	\det
	\begin{blockarray}{cccccc}
		& \mv & \mw & \partial_\me \mw / \omega & \partial_\me \mv / \omega & \text{other}\\
		\begin{block}{c(ccccc)}
  		\mw &	0 & 0 & 1 & 0 & \ast  \\
 		\mv &	 0 & 0 & 0 & 1 & \ast \\
 		\partial_\me \mv/\omega & 1 & - \cos(\ell_\me \omega) & \sin(\ell_\me \omega) & 0 & 0  \\
 		\partial_\me \mw/\omega & 0 & \sin(\ell_\me \omega) & \cos(\ell_\me \omega) & 1 & 0  \\
 		\text{other} &	 \ast & \ast & 0 & 0 & \ast \\
		\end{block}
	\end{blockarray}
	\end{align}
	and
	\begin{align}
	\label{eq:det_oriented_2}
	\det
	\begin{blockarray}{cccccc}
		& \mv & \mw & \partial_\me \mw / \omega & \partial_\me \mv / \omega & \text{other}\\
		\begin{block}{c(ccccc)}
  		\mw &	0 & 0 & 1 & 0 & \ast  \\
 		\mv &	 0 & 0 & 0 & 1 & \ast \\
 		\partial_\me \mv/\omega & \sin(\ell_\me \omega) & 0 & 1 & \cos(\ell_\me \omega) & 0  \\
 		\partial_\me \mw/\omega & - \cos(\ell_\me \omega) & 1 & 0 & \sin(\ell_\me \omega) & 0  \\
 		\text{other} &	 \ast & \ast & 0 & 0 & \ast \\
		\end{block}
	\end{blockarray}
	\end{align}
	Since the matrix in \eqref{eq:det_oriented_1} results from the one in	\eqref{eq:det_oriented_2} by multiplication from the left with the matrix 
	\[
	\begin{pmatrix}
		1 & & & &  \\
		& 1 & & & \\
		  & & 
		  \begin{matrix}
		  \sin(\ell_\me \omega) & - \cos(\ell_\me \omega)\\
		  \cos(\ell_\me \omega) & \sin(\ell_\me \omega)\\
		  \end{matrix}
		  & &\\
		& & & \dots & \\ 
		 & & & & 1\\
		\end{pmatrix},
	\]
	which itself has determinant \(1\), (a) follows.
	
	For (b), let us denote the original edge $\tilde \me = \mathsf u \mw$ which is subdivided in two edges $\me = \mathsf u \mv$ and $\mf = \mv \mw$ by inserting a new vertex $\mv$ on $\tilde \me$.
	We describe a sequence of steps which reduce the secular function of the extended graph to the one of the original one. 
	After simultaneous permutations of lines and columns, the matrix corresponding to the graph with extra vertex will be
	\begin{equation}
	\label{eq:huge_determinant}
	\begin{blockarray}{ccccccccc}
		& \mathsf u & \mv & \mw & \partial_\me {\mathsf u} / \omega & \partial_\me \mv / \omega & \partial_\mf \mv / \omega & \partial_\me \mw / \omega & \text{other}\\
	\begin{block}{c(cccccccc)}
  	\mathsf u 					&  0   & 0  & 0   & 1   &  0 & 0 & 0 & \ast \\
	\mv		  					&  0   & 0  & 0   & 0   & 1 & 1 & 0 & 0  	\\
	\mw		  					&  0   & 0  & 0   & 0   & 0 & 0 & 1 & \ast  \\
	\partial_\me {\mathsf u}/\omega	& \sine& 0  & 0   &\cose& 1 & 0 & 0 & 0	    \\
	\partial_\me \mv / \omega 	&-\cose& 1  & 0   &\sine& 0 & 0 & 0 & 0     \\
	\partial_\mf \mv / \omega 	&  0   &\sinf& 0  & 0   & 0 &\cosf& 1 & 0   \\
	\partial_\me \mw / \omega	&  0   &-\cosf& 1 & 0  & 0 &\sinf & 0 & 0   \\
	\text{other}				& \ast & 0 & \ast   & 0   & 0 & 0 & 0 & \ast  \\
	\end{block}
	\end{blockarray}.
	\end{equation}
	We now perform the following sequence of manipulations, all of which leave the determinant of~\eqref{eq:huge_determinant} invariant:
	\begin{itemize}
	\item
	subtract the $\partial_\mf \mv / \omega$ column from the $\partial_\me \mv / \omega$ column,
	\item
	develop along the $\mv$ row (which now has only one nonzero entry),
	\item
	in the resulting minor subtract $\sinf$ times the $\partial_\me \mv / \omega$ row from the $\partial_\mf \mv / \omega$ row and add it $\cosf$ times to the $\partial_\mf \mw / \omega$ row,
	\item 
	develop along the $\mv$ column (which has only one nonzero entry),
	\item
	in the resulting minor add $\cose$ times the $\partial_\me {\mathsf u} / \omega$ row to the $\partial_\mf \mv / \omega$ row and $\sine$ times $\partial_\me {\mathsf u} / \omega$ to the $\partial_\me {\mathsf u} / \omega$ row,
	\item
	develop along the $\partial_\me \mv / \omega$ column (which has only one nonzero entry).
	\end{itemize}
	This simplifies the determinant of~\eqref{eq:huge_determinant} to
	\begin{align*}
	\det
	\begin{blockarray}{cccccc}
		& \mv & \mw & \partial_\me \mw / \omega & \partial_\me \mv / \omega & \text{other}\\
		\begin{block}{c(ccccc)}
  {\mathsf u} &	0 & 0 & 1 & 0 & \ast  \\
 		\mw   &	 0 & 0 & 0 & 1 & \ast \\
 		\partial_\me \mv/\omega & \sin(\ell_{\tilde\me} \omega) & 0 & 1 & \cos(\ell_{\tilde \me} \omega) & 0  \\
 		\partial_\me \mw/\omega & - \cos(\ell_{\tilde \me} \omega) & 1 & 0 & \sin(\ell_{\tilde \me} \omega) & 0  \\
 		\text{other} &	 \ast & \ast & 0 & 0 & \ast \\
		\end{block}
	\end{blockarray}
	\end{align*}
	where we used that due to $\ell_{\tilde \me} = \ell_\me + \ell_\mf$ we have	
	\begin{align*}
	\cos(\ell_{\tilde \me} \omega)
	&=
	\cose \cosf - \sine \sinf
	,
	\quad
	\text{and}
	\\
	\cos(\ell_{\tilde \me} \omega)
	&=
	\sine \cosf + \cose \sinf.
	\end{align*}
	This determinant corresponds to the structure of the original graph.
\end{proof}

Nevertheless, there is a common problem with secular matrices: the isomorphism between $\mathrm{Eig}(-\Delta,\lambda)$ and $\ker(A_\omega)$ in Lemma~\ref{lem:isomorphism} does not persist at $\omega = 0$, i.e. for the matrix
\[
	A_0
	=
	\begin{pmatrix}
		\bigzero
	&
	\rvline
	&
		\begin{matrix}
		\dots & 1 & \dots & 1 & \dots \\
		1 & \dots 1 & \dots & 1 \\
		&&\dots&&
		\end{matrix}
	\\
	\hline
		\begin{matrix}
		1 & - 1\\
		  & 0\\
		  & \dots\\
		\end{matrix}
	&
	\rvline
	&		
		\begin{matrix}
		0 & \\
		1 & 1 \\
		& \dots\\
		\end{matrix}
	\end{pmatrix}.
\] 
Indeed, bearing in mind that $\lim_{\omega \to 0} \sin(\ell_\me \omega) / \omega = \ell_\me$, the eigenspace corresponding to the eigenvalue $\lambda = 0$ would correspond to the kernel of the different matrix
\begin{equation}
	\label{eq:harmonic}
	\begin{pmatrix}
		\bigzero
	&
	\rvline
	&
		\begin{matrix}
		\dots & 1 & \dots & 1 & \dots \\
		1 & \dots 1 & \dots & 1 \\
		&&\dots&&
		\end{matrix}
	\\
	\hline
		\begin{matrix}
		1 & - 1\\
		  & 0\\
		  & \dots\\
		\end{matrix}
	&
	\rvline
	&		
		\begin{matrix}
		\ell_\me & \\
		1 & 1 \\
		& \dots\\
		\end{matrix}
	\end{pmatrix}.
\end{equation}
The matrix in~\eqref{eq:harmonic} describes \emph{edgewise linear}, that is harmonic functions on the graph and has \emph{one-dimensional kernel} whereas the kernel of $A_0$ will in general have higher dimension, see Proposition~\ref{prop:ker_A_0} below.
Fortunately, if the graph is a tree, the matrix $A_0$ retains useful properties which are reminiscent of the ground state of the Laplacian: 
\begin{lemma}
	\label{lem:kernel}
	If $\mathcal G$ is a metric tree with standard conditions at all vertices (i.e \(\mV_D=\emptyset\)), then the matrix $A_0$ has one-dimensional kernel and
	\[
		\ker A_0
		=
		\operatorname{Span}
		\{
			x_0
		\},
	\quad
	\text{where}
	\quad
	x_0
	:=
	\frac{1}{\sqrt{\lvert \mV \rvert}}
		\begin{pmatrix}
		1\\
		\dots
		\\
		1\\
		0\\
		\dots
		\\
		0\\			
		\end{pmatrix}
		.
	\]
\end{lemma}

\begin{proof}
	Clearly $x_0 \in \ker A_0$. 
	Furthermore, it follows from $A_0$ that any solution to $A_0 x = 0$ has the same entry in all of its $\lvert \mV \rvert$ many coordinates.
	It remains to see that all elements in the $\ker A_0$ vanish on the last $2 \lvert \mE \rvert$ entries.
	This is due to the tree structure: 
	The lower left submatrix of $A_0$ implies that entries corresponding to derivatives at both end points of an edge must be identical while the standard conditions, encoded in the top right submatrix of $A_0$, imply that all entries corresponding to outgoing derivatives at all leaves must vanish.
	Inductively, one can now remove leaves from the graph and eventually finds that all of the last $2 \lvert \mE \rvert$ entries of $x \in \ker A_0$ vanish.
\end{proof}

More generally, the matrix $A_0$ contains information on the topological structure of the unterlying graph:
\begin{proposition}
	\label{prop:ker_A_0}
	We have on connected graphs 
		\[\dim \ker A_0=\beta(\mathcal G)+1\]
	if  \(\mV_D=\emptyset\) and
		\[\dim \ker A_0=\# \mV_D +\beta(\mathcal G)-1\]
	if \(\#\mV_D\geq 1\) where $\beta(\mathcal G) = \# \mE - \# \mV + 1 \geq 0$ is the first Betty number of the graph.
\end{proposition}

\begin{proof}
	We only prove the first statement, since the second one can be shown using similar arguments. We first observe that \(\ker A_0(\mathcal G)\) may decomposed as
		\[\ker A_0(\mathcal G)=V(\mathcal G)\oplus W(\mathcal G)\]
	where
		\[V(\mathcal G)=\mathrm{Span} \begin{pmatrix}
		1\\
		\dots
		\\
		1	
		\end{pmatrix}\subset \mathbb R^{|V|}\]
	and \(W(\mathcal G)\subset \mathbb R^{2|\mE|}\) is the space of vectors \(\{x(\mv,\me)\}_{\mv \in \mV, \me \in \mE_\mv}\) satisfying
		\[x(\mv,\me)= -x(\mw,\me)\]
	for all edges \(\me\in\mE\) connecting two vertices \(\mv\) and \(\mw\) and
		\[\sum_{\me\in\mE_\mv} x(\mv,\me)=0\]
	for all vertices \(\mv\in\mV\). We thus have to show that \(\dim W(\mathcal G)=\beta(\mathcal G)\). To do so we proceed by induction over \(\beta\).  For \(\beta=0\) the statement follows from Lemma \ref{lem:kernel}. For \(\beta >0\), we may cut \(\mathcal G\) through some appropriately chosen vertex \(\mv\) in \(\mathcal G\), so that \(\mv\) is split into two vertices \(\mv_1\) and \(\mv_2\) resulting in a new connected graph \(\tilde {\mathcal G}\) with first Betti number \(\beta(\tilde{\mathcal G})=\beta(\mathcal G)-1\), or equivalently
	\begin{equation}\label{eq:induction-step-betti-number}
		\beta(\mathcal G)=\beta(\tilde{\mathcal G})+1.
	\end{equation}
	Moreover, by definition of \(W(\mathcal G)\), we have
		\[W(\tilde{\mathcal G})=\left\{\{x(\mv,\me)\}_{\mv \in \mV, \me \in \mE_\mv}\in W(\mathcal G)~\big|~\sum_{\me\in\mE_{\mv_1}}x(\me,\mv)=0\right\},\]
	from which we infer that \(W(\tilde{\mathcal G})\) has codimension \(1\) in \(W(\mathcal G)\).
	\begin{equation}\label{eq:induction-step-dimension-W-G}
		\dim W(\mathcal G)=\dim W(\tilde{\mathcal G})+1.
	\end{equation}
	Thus, the stated equality \(\dim W(\mathcal G)=\beta(\mathcal G)\) follows by induction over \(\beta(\mathcal G)\) using \eqref{eq:induction-step-betti-number} and \eqref{eq:induction-step-dimension-W-G}.
\end{proof}

\begin{remark}
	\label{rem:why_only_trees}
	Proposition~\ref{prop:ker_A_0} can be seen as an illustration for the reason why we cannot go beyond trees if there are standard conditions at all vertices.
	Below, we will strongly rely on the fact that on trees, $\det A_\omega$ has a zero of first order at $\omega = 0$.
	If $\mG$ is not a tree, then it has non-zero Betti number, whence $A_0$ will have a higher dimensional kernel and $\omega \mapsto \det A_\omega$ will have a zero of higher order at $\omega = 0$.
	Therefore, the case of non-tree graphs with only standard conditions remains open.
\end{remark}

Now comes a subtle point:
In general, if $A_\omega$ has $k$-dimensional kernel, then the secular function $\omega \mapsto \det A_\omega$ must have a zero of \emph{at least} $k$-th order at $\omega$.
We have not been able to prove that the order of the zero is equal to the dimension of the kernel, even though this seems very plausible. 
This is actually one disadvantage of our approach compared to more classical forms of the secular function where it is known that the order of positive zeros coincide with the dimension of the eigenspace~\cite{BolteE-09}.
On metric trees, the matrix $A_0$ does have one-dimensional kernel, but we would like to use that $\det A_\omega$ has a zero of order one at $\omega = 0$.
This is why we need the following proposition:

\begin{proposition}
	\label{prop:determinant_non_zero_derivative}
	If \(\mathcal G\) is a metric tree
	\[
	\frac{\partial}{\partial\omega}_{\mid\omega = 0} \big[\det A_\omega (\mathcal G)\big]
	=
	- \sum_{\me \in \mE} \ell_\me
	.
	\]
	In particular, the map $\omega \mapsto \det A_\omega(\mathcal G)$ has a zero of order one at $\omega = 0$.
\end{proposition}

The proof of Proposition~\ref{prop:determinant_non_zero_derivative} relies on the following two lemmas.
We believe that Lemma~\ref{lem:determinant_formulas} is interesting in its own right and a strength of our secular function since it establishes a connection between a secular function and certain graph surgeries.
Recall that $A_\omega(\mathcal{G}, \mv)$ denotes the secular function of the metric graph $\mathcal{G}$ with a Dirichlet condition at the vertex $\mv$ and standard conditions elsewhere.

\begin{lemma}
	\label{lem:determinant_formulas}
	Let $\mathcal G$ be a compact metric graph and $\mathcal G^+$ be the graph obtained by attaching a new edge $\me = \mv \mw$ between a vertex $\mv$ of $\mathcal G$  with standard conditions and a new vertex $\mw$.
	\begin{enumerate}[(i)]
	\item
	If we impose standard conditions at $\mw$, then
	\begin{equation}
	\label{eq:add_edge_Kirchhoff}
	\det A_\omega(\mathcal G^+,\mV_D)
	=
	-
	\sin(\ell_\me \omega) 
	\det A_\omega (\mathcal G , \mV_D\cup\{\mv\}) 
	+
	\cos(\ell_\me \omega) 
	\det A_\omega(\mathcal G,\mV_D)
	.
\end{equation}
	\item
	If we impose Dirichlet conditions at $\mw$, then 
	\begin{equation}
	\label{eq:add_edge_Neumann}
	\det A_\omega(\mathcal G^+ , \mV_D\cup\{\mw\})
	=
	\sin(\ell_\me \omega) \det A_\omega (\mathcal G,\mV_D) 
	+
	\cos(\ell_\me \omega) 
	\det A_\omega(\mathcal G , \mV_D\cup\{\mv\})
	\end{equation}
	\item If we impose Dirichlet vertex conditions at \(\mv\) and \(\mathcal G\setminus\{ \mv \}\) is disconnected
	\begin{equation}
	\label{eq:Dirichlet_disconnected}
	\det A_\omega(\mathcal G , \mv)=\prod_{i=1}^n
	\det A_\omega(\mathcal G_i,  \mv)
	\end{equation}
	where \(\mathcal G_1,\ldots,\mathcal G_n\) denote the components of \(\mathcal G\setminus \{ \mv \}\) with a copy of the vertex $\mv$ added again to each component.
\end{enumerate}
\end{lemma}
\begin{proof}
Simultaneous permutations of rows and columns do not change the determinant, so we may assume that the first four rows and columns of the matrix correspond to $ (\mw,\mv,\partial_\me \mw/\omega, \partial_\me \mv/\omega)$.
We calculate
\begin{align*}
	\det 
	A_\omega(\mathcal G^+)
	&=
	\det
	\begin{blockarray}{cccccc}
		& \mw & \mv & \partial_\me \mv / \omega & \partial_\me \mw / \omega & \text{other}\\
		\begin{block}{c(ccccc)}
  		\mw &	0 & 0 & 0 & 1 & 0  \\
 		\mv &	 0 & 0 & 1 & 0 & \ast \\
 		\partial_\me \mv/\omega & 1 & - \cos(\ell_\me \omega) & \sin(\ell_\me \omega) & 0 & 0  \\
 		\partial_\me \mw/\omega & 0 & \sin(\ell_\me \omega) & \cos(\ell_\me \omega) & 1 & 0  \\
 		\text{other} &	 0 & \ast & 0 & 0 & \ast \\
		\end{block}
	\end{blockarray}
	\\
	&=
	\det
	\begin{blockarray}{cccc}
		 & \mv & \partial_\me \mv/\omega & \text{other}\\
		\begin{block}{c(ccc)}
   		\mv & 0 & 1 & \ast \\
 		\partial_\me \mw / \omega & \sin(\ell_\me \omega) & \cos(\ell_\me \omega) & 0  \\
 		\text{other} & \ast & 0 & \ast \\
		\end{block}
	\end{blockarray}\\
	&=
	-
	\sin (\ell_\me \omega)
	\det
	\underbrace{\begin{pmatrix}
		1 & \ast \\
		0 & \ast \\
	\end{pmatrix}
	}_{= \det A_\omega(\mathcal G, \mv )}
	+
	\cos(\ell_\me \omega)
	\det A_\omega(\mathcal G)
\end{align*}
where we developed the determinant along the $\mw$ row in the first step and along the $\partial_\me \mw / \omega$ row in the second step.
This shows~\eqref{eq:add_edge_Kirchhoff}.
For~\eqref{eq:add_edge_Neumann}, we calculate
\begin{align*}
	\det 
	A_\omega(\mathcal G^+, \mw)
	&=
	\det
	\begin{blockarray}{cccccc}
		& \mw & \mv & \partial_\me \mv / \omega & \partial_\me \mw  / \omega& \text{other}\\
		\begin{block}{c(ccccc)}
  		\mw &	1 & 0 & 0 & 0 & 0  \\
 		\mv &	 0 & 0 & 1 & 0 & \ast \\
 		\partial_\me \mv / \omega & 1 & - \cos(\ell_\me \omega) & \sin(\ell_\me \omega) & 0 & 0  \\
 		\partial_\me \mw / \omega & 0 & \sin(\ell_\me \omega) & \cos(\ell_\me \omega) & 1 & 0  \\
 		\text{other} &	 0 & \ast & 0 & 0 & \ast \\
		\end{block}
	\end{blockarray}
	\\
	&=
	\det
	\begin{blockarray}{cccc}
		 & \mv & \partial_\me/\omega \mw & \text{other}\\
		\begin{block}{c(ccc)}
   		\mv & 0 & 1 & \ast \\
 		\partial_\me \mw/\omega & - \cos(\ell_\me \omega) & \sin(\ell_\me \omega) & 0  \\
 		\text{other} & \ast & 0 & \ast \\
		\end{block}
	\end{blockarray}\\
	&=
	\cos(\ell_\me \omega) 
	\underbrace{
	\det
	\begin{pmatrix}
		1 & \ast \\
		0 & \ast \\
	\end{pmatrix}
	}_{= \det A_\omega(\mathcal G , \mv)}
	+
	\sin (\ell_\me \omega)
	\det A_\omega(\mathcal G).
\end{align*}

Finally, \eqref{eq:Dirichlet_disconnected} follows by similar manipulations using again that simultaneous permutations of rows and columns do not affect the determinant
\begin{align*}
	\det
	\begin{blockarray}{ccc}
		 & \mv & \text{other}\\
		\begin{block}{c(cc)}
   		\mv & 1 & 0 \\
 		\text{other} & \ast & \ast \\
		\end{block}
	\end{blockarray}
	&=
	\det
	\begin{blockarray}{ccc}
		 & \mv & \text{other}\\
		\begin{block}{c(cc)}
   		\mv & 1 & 0 \\
 		\text{other} & 0 & \ast \\
		\end{block}
	\end{blockarray}
	=
	\det
	\begin{blockarray}{cccccc}
		& \mv_1 & \mv_2 & \dots & \mv_n & \text{other}\\
		\begin{block}{c(ccccc)}
  		\mv_1 & 1 &   &   & 0 & 0  \\
 		\mv_2 &   & 1 &   &  & 0 \\
 		    &   &   & 1 & \dots & 0  \\
 		\mv_n &   &   &   & 1 & 0 \\
 		\text{other} & 0 & 0 & 0 & 0 & \ast \\
		\end{block}
	\end{blockarray}
	\\
	&=
	\det
	\begin{pmatrix}
		A_\omega (\mathcal G_1, \mv ) &\rvline&  0 &\rvline& 0 \\
		\hline
		 0  &\rvline&  \dots &\rvline& 0 \\	
		\hline
		0   &\rvline&  0 &\rvline& A_\omega (\mathcal G_n , \mv) \\			 
	\end{pmatrix}
	=
	\prod_{i=1}^n
	\det A_\omega(\mathcal G_i, \mv).
	\qedhere
\end{align*}

\end{proof}

\begin{lemma}
	\label{lem:formula_for_Neumann}
	Let $\mathcal G$ be a tree with Dirichlet boundary conditions at exactly one leaf $\mv$ and standard conditions elsewhere.
	Then
\begin{equation}
	\label{eq:Dirichlet_determinant_at_0}
	\det A_0 (\mathcal G, \mv)
	=
	1
	.
\end{equation}
\end{lemma}

\begin{proof}
	We proceed by induction over $\lvert \mE \rvert$.
\\
If $\lvert \mE \rvert = 1$, then there are exactly two leaves and
\[
	\det A_\omega
	=
	\det
	\begin{pmatrix}
	1 & 0 & 0 & 0\\
	0 & 0 & 1 & 0\\
	1 & - \cos(\omega \ell_\me) & \sin(\omega \ell_\me) & 0 \\
	0 & \sin(\omega \ell_\me) & \cos(\omega \ell_\me) & 1 \\
	\end{pmatrix}
	=
	\cos(\omega \ell_\me)
	.
\]
For the induction step, let us add to a tree $G$ a new edge $\mv \mw$ with Dirichlet boundary conditions at the leaf $\mw$ and call the new tree $G^+$.
Then, by~\eqref{eq:add_edge_Neumann}
\begin{align*}
	\det A_0 (\mathcal G^+, \mw)
	&=
	\sin(0) \det A_0(\mathcal G)
	+
	\cos(0)
	\underbrace{\prod_i \det A_0(\mathcal G_i, \mv)}_{= 1}
	=
	1.
	\qedhere
\end{align*}
\end{proof}

We are now ready for the proof of Proposition~\ref{prop:determinant_non_zero_derivative}.

\begin{proof}[{Proof of Proposition~\ref{prop:determinant_non_zero_derivative}}]

We proceed by induction over $\lvert \mE \rvert$.
\\
If $\lvert \mE \rvert = 1$, then
\[
	\det A_\omega
	=
	\begin{pmatrix}
	0 & 0 & 0 & 1\\
	0 & 0 & 1 & 0\\
	1 & - \cos(\omega \ell_\me) & \sin(\omega \ell_\me) & 0 \\
	0 & \sin(\omega \ell_\me) & \cos(\omega \ell_\me) & 1 \\
	\end{pmatrix}
	=
	-\sin( \omega \ell_\me)
\]
which has derivative
\[
	\frac{\partial}{\partial\omega}
	\det A_\omega \mid_{\omega = 0}
	=
	-\ell_\me 
	\cos ( \omega \ell_\me ) \mid_{\omega = 0}
	=
	- \ell_\me .
\]
For the induction step, we call $\mathcal G^+$ the tree obtained by adding another edge $\tilde 
\me = \mv \mw$ with standard conditions at $\mv$ and $\mw$ to the graph $\mathcal G$ and find by~\eqref{eq:add_edge_Kirchhoff} and~\eqref{eq:Dirichlet_determinant_at_0}
\begin{align*}
	\frac{\partial}{\partial\omega}
	\left[
		\det A_\omega (\mathcal G^+)
	\right]_{\omega = 0}
	&=
	\frac{\partial}{\partial\omega}
	\left[
	- 
	\sin(\ell_{\tilde \me} \omega) 
	\det A_\omega (\mathcal G , \mv) 
	+
	\cos(\ell_{\tilde \me} \omega) 
	\det A_\omega(\mathcal G)
	\right]_{\omega = 0}
	\\
	&=
	- \ell_{\tilde \me}
	\underbrace{\det A_0 (\mathcal G, \mv )}_{=1}
	+
	\frac{\partial}{\partial\omega}
	\left[
		\det A_\omega(\mathcal G)
	\right]_{\omega = 0}
	\\
	&= 
	-\ell_{\tilde \me}
	-
	\sum_{\me \in \mE_\mathcal G} \ell_\me
	=
	-\sum_{\me \in \mE_{\mathcal G^+}} \ell_\me.
	\qedhere
\end{align*}
\end{proof}

From now on, the tree $\mathcal G$ is fixed and we use the  analytical properties of the secular function to prove Theorem~\ref{thm:main_tree}.
Denote by \(\lVert\cdot\rVert_{\mathrm{F}}\) the Frobenius norm on \(\RR^{\lvert \mV \rvert + 2 \lvert E \rvert\times\lvert \mV \rvert + 2 \lvert E \rvert}\) and let \(\lvert\cdot\rvert\) denote the Euclidean norm on \(\RR^{\lvert \mV \rvert + 2 \lvert \mE \rvert}\). The main tool in the proof of Theorem \ref{thm:main_tree} is the following

\begin{lemma}
	\label{lem:sequence}
	There exists a sequence $(\omega_n)_{n \in \NN}$ with $\det A_{\omega_n} = 0$ for all \(n\in\NN\), such that $\omega_n \to \infty$ and $A_{\omega_n} \rightarrow A_0$ as \(n\rightarrow \infty\).
\end{lemma}

\begin{proof}
	It suffices to show that for all \(C>0\) and all \(0<\varepsilon<1\) there exists some \(\hat\omega\geq C\), so that \(\det A_{\hat \omega}=0\) and \(\|A_{\hat\omega}-A_{0}\|_F\leq M \varepsilon\), where \(M>0\) is some constant that does not depend on \(\varepsilon\) or \(C\). From this the desired sequence may be constructed recursively.\\
	We consider the functions \(F,h:\mathbb R\rightarrow \mathbb R\) with \(F(\omega):=\det A_\omega\) and
		\[h(\omega) := F(\omega)^2 + (F'(\omega) - F'(0))^2 + \|A_\omega-A_0\|_F^2\]
	for \(\omega\in\mathbb R\).	 
	They are clearly analytic and by Lemma~\ref{lem:Taylor}, there exists $\delta > 0$ such that, if 
	\begin{equation}
	\label{eq:maximum_fs_bounded}
	\max
	\left\{
		\lvert F(\omega) \rvert, 
		\lvert F'(\omega) - F'(0) \rvert 
	\right\}
	\leq 
	\delta,
	\end{equation}
	then $F$ must have a zero in an $\epsilon^2$-neighbourhood of $\omega$.
	Moreover, the function $h$ is a nonnegative trigonometric polynomial and thus almost periodic with $h(0) = 0$, see for instance~\cite{Bohr-1925} for a reference. Thus, there exist some \(\omega_0\geq C+\varepsilon^2\) with
	\begin{equation}\label{eq:estimate-almostperiod-applied}
	h(\omega)
	\leq
	\min
	\left\{
	\delta^2,
	\epsilon^2
	\right\}
	\end{equation}
	In particular, \(\omega_0\) satisfies \eqref{eq:maximum_fs_bounded}, which in turn yields that there is some \(\hat\omega\in [\omega_0-\varepsilon^2,\omega_0+\varepsilon^2]\) with $F(\hat\omega) = \det A_{\hat\omega} = 0$. Using \(|\hat\omega-\omega_0|\leq \varepsilon^2\) and \eqref{eq:estimate-almostperiod-applied} we estimate
	\begin{align*}
		\|A_{\hat\omega}-A_0\|_F^2  \leq 2(\|A_{\hat\omega}-A_{\omega_0}\|_F^2 +\|A_{\omega_0}-A_0\|_F^2) \leq 2\left(\sup_{\omega\in\mathbb R}\Big|\frac{\partial}{\partial \omega}\|A_{\omega}\|_F^2\Big|+1\right)\varepsilon^2.
	\end{align*}
	Setting \(M:=\sqrt{2\left(\sup_{\omega\in\mathbb R}\big|\frac{\partial}{\partial \omega}\|A_{\omega}\|_F^2\big|+1\right)}\) proves the claimed statement.
	\end{proof}
\begin{proof}[Proof of Theorem~\ref{thm:main_tree}]
	Take the sequence $(\omega_n)_{n \in \NN}$ from Lemma~\ref{lem:sequence}. Since \(\det A_{\omega_n}=0\) holds for all \(n\in\mathbb N\), there exists a sequence \((x_n)_{n\in\mathbb N}\) with $x_n \in \ker A_{\omega_n}$ and \(\lvert x_n\rvert=1\) for all \(n\in\mathbb N\). By compactness, we may assume -- after passing to a subsequence -- that $(x_n)_{n\in\mathbb N}$ converges to some $x_\infty$ with \(\lvert x_\infty\rvert=1\). Note that, by choice of \((\omega_n)_{n\in\mathbb N}\), the sequence \((A_{\omega_n})_{n\in\mathbb N}\) converges to \(A_0\). We obtain
		\[A_0x_\infty=\lim_{n\rightarrow\infty}A_{\omega_n}x_n=0.\]
	Now, since \(x_\infty\) is a normalized element of the kernel of \(A_0\), Lemma \ref{lem:kernel} yields
		\[x_\infty=\pm \frac{1}{\sqrt{\lvert V \rvert}}
		\begin{pmatrix}
		1\\
		\dots
		\\
		1\\
		0\\
		\dots
		\\
		0\\			
		\end{pmatrix}
		.
	\]
	Therefore, for sufficiently large \(n\in\mathbb N\), the first \(| \mV |\) entries of \(x_n\) must differ from \(0\). This, in turn, means that for sufficiently large \(n\) the eigenfunction \(\psi_n\) of \(\Delta\) that corresponds to \(x_n\) in the sense of Lemma \ref{lem:isomorphism} does not vanish in any of the vertices of \(\mathcal G\) and is thus supported on the whole graph \(\mathcal G\). This completes the proof.
\end{proof}
\section{Metric graphs with at least one Dirichlet vertex}
\label{sec:Dirichlet}
In this section, we prove Theorem~\ref{thm:main_Dirichlet} for graphs with at least one Dirichlet vertex. The proof relies on the fact that the eigenfunctions corresponding lowest eigenvalue of the Laplacian does not vanish on \(\mathcal G \setminus\mV_D\) if \(\mathcal G\setminus\mV_D\) is connected (see \cite{Kurasov-19}). Moreover it uses the following

\begin{lemma}\label{lem:sign-change-at-ground-state}
	Suppose that \(\mV_D\) is nonempty and \(\mathcal G\setminus\mV_D\) is connected. Then, $\omega \mapsto \det A_\omega$ changes sign in every neighbourhood of $\omega_0$ where $\omega_0^2 > 0$ is the ground state eigenvalue.
\end{lemma}

\begin{proof}
	Let $\mv$ be a Dirichlet vertex , which we may always assume to have degree one, and which is connected to a vertex $\mw$ with standard conditions via the edge $\me$. For our proof it will be helpful to consider the dependence of the ground state eigenvalue of the length of the edge \(\me\). For that purpose let \(\mathcal G_{s}\) denote the metric graph with the same combinatorial structure as \(\mathcal G\) where the edge \(\me\) has the new length \(s>0\) and the other edges have the same length as in \(\mathcal G\). With this notation we have \(\mathcal G_{\ell_\me}=\mathcal G\).\\
	It is well-known that locally around \(\ell_\me\) there is a differentiable map \(s\mapsto \omega(s)\) with $\omega(\ell_\me) = \omega_0$, such that $\det A_\omega (\mathcal G_s,\mV_D) = 0$ if and only if $\omega = \omega(s)$. 
	Moreover, it follows from a Hadamard-type formula that said map is strictly decreasing in \(s\), see for instance \cite[Section 3.1.]{BerkolaikoK-12-book} or \cite{Friedlander-05} for reference. We conclude that the graph of $s\mapsto\omega(s)$ splits a neighbourhood of $(\ell_\me, \omega_0)$ into two disjoint connected domains on each of which $(s,\omega) \mapsto \det A_\omega(\mathcal G_s,\mV_D)$ is sign-definite.\\	
	Thus it sufficices to see that the partial derivative of $\det A_\omega(\mathcal G_s)$ in the \(s\)-direction is non-vanishing, because then we have found that $\det A_\omega(\mathcal G_s,\mV_D)$ also changes its signs in the \(\omega\)-direction as stated.
	
	For this purpose, note that Lemma~\ref{lem:determinant_formulas} (ii) yields
	\begin{align*}
	\det A_\omega(\mathcal{G}_s,\mV_D)
	&=
	\cos (s \omega)
	\det A_\omega(\mathcal{G}^-, (\mV_D\setminus\{\mv\})\cup\mw)
	+
	\sin(s \omega)
	\det A_\omega(\mathcal{G}^-,\mV_D\setminus\{\mv\})
	\\
	&=
	\cos (s \omega) X(\omega) + \sin (s \omega) Y(\omega)
	\end{align*}
	where $\mathcal{G}^-$ stands for the graph with the edge $\me$ removed.
	This expression is $0$ at $\omega_0$ by assumption, but we know that $X(\omega_0) \neq 0$, since the ground state (the first positive zero) of the smaller graph must strictly above the ground state of the larger graph.
	Now,
	\[
	\frac{\partial}{\partial{s}}
	\det A_\omega(\mathcal{G}_s,\mV_D)
	=
	s
	\left[
		- \sin (s \omega)
		X(\omega)
		+
		\cos (s \omega)
		Y(\omega).
	\right]
	\]
	So, if this derivative were zero at \(s=\ell_\me\) and $\omega = \omega_0$, we would obtain
	\[
	\begin{pmatrix}
	X(\omega_0)
	\\
	Y(\omega_0)
	\end{pmatrix}
	\in
	\ker 
	\begin{pmatrix}
	\cos (\ell_\me \omega) & \sin (\ell_\me \omega) \\
	-\sin (\ell_\me \omega) & \cos (\ell_\me \omega) \\
	\end{pmatrix}
	=
	\left\{ \begin{pmatrix} 0 \\ 0 \end{pmatrix} \right\}
	\]
	contrary to $X(\omega_0) \neq 0$.
\end{proof}

The rest of the proof of Theorem~\ref{thm:main_Dirichlet} is analogous to the proof of Theorem~\ref{thm:main_tree} on trees with two slight modifications:
we first did not prove that $\omega_0$ is a zero of first order, but merely a zero of odd order of $\det A_\omega$. However, this is still covered by Lemma~\ref{lem:Taylor}. Second, we use the previously mentioned result from \cite{Kurasov-19} that states that the eigenfunctions corresponding to the lowest eigenvalue of the Laplacian on \(\mathcal G\) does not vanish on \(\mathcal G\setminus\{\mV_D\}\) which yields that a normalized vector in \(\ker A_{\omega_0}(\mathcal G,\mV_D)\) is non-zero on \(\mV\setminus\mV_D\).

\section{Metric graphs with \(\delta\)-couplings and edgewise constant potentials}
In this section we prove Theorems \ref{thm:main_delta} and \ref{thm:main_potential}. Throughout this section we suppose that \(\mathcal G\) is a tree or that \(\mathcal G\) has at least one Dirichlet vertex. 
If \(\mathcal G\) is a tree, we put \(\omega_0:=0\), if there is at least one Dirichlet vertex, let \(\omega_0>0\) be the frequency corresponding to the ground state of \(-\Delta_\mathcal G\).  
As in the previous sections let \(A_\omega\) for \(\omega\in\mathbb R\) denote the matrix defined in \eqref{eq:definition-A-omega}. The main tool in the proofs of Theorems \ref{thm:main_delta} and \ref{thm:main_potential} is the following pertubation lemma that generalizes Lemma \ref{lem:sequence}.
\begin{proposition}\label{lem:pertubation}
	Let \(a>0\) and, for each \(\omega\geq a\), let \(P_\omega\in\mathbb R^{|\mV|+2|\mE|\times|\mV|+2|\mE|}\) be a matrix such that the map \(\omega\mapsto P_\omega\) is bounded and infinitely differentiable on \([a,\infty)\) in every entry such that all of its derivatives are bounded on \([a,\infty)\). Then, for the perturbed matrices
		\[B_\omega:= A_\omega+\omega^{-1} P_\omega,\quad \omega\geq a,\]
	there exists an increasing sequence \((\omega_n)_{n\in\mathbb N}\) in \([a,\infty)\) with \(\omega_n\rightarrow \infty\), so that \(\det B_{\omega_n}=0\) for all \(n\in\mathbb N\) and \(B_{\omega_n}\rightarrow A_{\omega_0}\) as \(n\rightarrow\infty\).
\end{proposition}
\begin{proof}
It suffices to show that for all \(C>0\) and all \(0<\varepsilon<1\) there exists some \(\hat\omega\geq C\), so that \(\det B_{\hat \omega}=0\) and \(\|B_{\hat\omega}-A_{\omega_0}\|_F\leq M \varepsilon\), where \(M>0\) is some constant that does not dependent on \(\varepsilon\) or \(C\).\\
We consider the function \(F:\mathbb R\rightarrow \mathbb R\) with \(F(\omega):=\det A_\omega\). By Lemma \ref{lem:sign-change-at-ground-state} and the analyticity of \(F\) we find some \(k\in\mathbb N_0\) so that \(F^{(j)}(\omega_0)=0\) for \(j\leq 2k\) and \(F^{(2k+1)}(\omega_0)\neq 0\). Next we consider the function \(G:[a,\infty)\rightarrow \mathbb R\) with \(G(\omega):=\det B_\omega\). 
By assumption, there exists a bounded and infinitely differentiable function \(R:[a,\infty)\rightarrow\mathbb R\) with bounded derivatives, so that
	\[G(\omega)=F(\omega)+\omega^{-1}R(\omega), \quad \omega\geq a.\]
In particular, all derivatives of \(G\) are bounded. Thus, following Lemma~\ref{lem:Taylor}, we may choose some \(\delta>0\), so that for \(\omega\geq a+\varepsilon^2\) the inequality
	\[
	\max 
		\left\{ 
		\lvert G(\omega) \rvert, 
		\lvert G'(\omega) \rvert, 
		\dots, 
		\lvert G^{(2 k)}(\omega) \rvert,
		\lvert G^{(2 k + 1)}(\omega) - F^{(2k+1)}(\omega_0) \rvert
		\right\}
		\leq
		\delta,
	\]
implies that \(G\) has a zero in \([\omega-\varepsilon^2,\omega+\varepsilon^2]\).
Now by boundedness of \(P_\omega\), \(R\), and its derivatives, we may assume that \(C>0\) is large enough so that the following inequalites hold for all $\omega \geq C$:
\begin{equation}\label{eq:estimate-perturbed-sec-function}
	|G^{(j)}(\omega)-F^{(j)}(\omega)|\leq \frac{\delta}{2}
	\quad
	\text{for all \(j=0,1,\ldots,2k+1\)},
\end{equation}
and 
\begin{equation}\label{eq:frob-estimate-perturbed-matrix}
	\|B_\omega-A_\omega\|_F\leq \varepsilon .
\end{equation}
Now, consider the function \(h:\mathbb R\rightarrow \mathbb R\) given by
	\[h(\omega):= \sum_{j=0}^{2k} F^{(j)}(\omega)^2+(F^{(2k+1)}(\omega)-F^{(2k+1)}(\omega_0))^2+\|A_\omega-A_0\|_F^2\]
for \(\omega\in\mathbb R\). This is a nonnegative trigonometric polynomial and thus almost periodic with \(h(0)=0\). Therefore, there exists some \(\omega_1\geq C+\varepsilon^2\) with
\begin{equation}\label{eq:estimate-almostperiod-applied-pert}
	h(\omega)\leq\max\left(\frac{\delta^2}{4},{\varepsilon^2}\right).
\end{equation}
Using \eqref{eq:estimate-perturbed-sec-function} and \eqref{eq:estimate-almostperiod-applied-pert} we obtain \(|G^{(2 k + 1)}(\omega_1) - F^{(2k+1)}(\omega_0)|\leq \delta\) and \(|G^{(j)}(\omega)|\leq \delta\) for \(j=0,1,\ldots,2k\). Thus, Lemma \ref{lem:Taylor} yields existence of some \(\hat\omega\in [\omega_1-\varepsilon^2,\omega_1+\varepsilon^2]\) with \(\det B_{\hat\omega}=G(\hat\omega)=0\). In particular, \(\hat\omega\geq C\) holds. Moreover, using \eqref{eq:frob-estimate-perturbed-matrix}, \eqref{eq:estimate-almostperiod-applied-pert} and \(|\omega_1-\hat\omega|\leq \varepsilon^2\) we obtain
	\begin{align*}
		\|B_{\hat\omega}-A_{\omega_0}\|_F^2 & \leq 3(\|B_{\hat\omega}-B_{\omega_1}\|_F^2+\|B_{\omega_1}-A_{\omega_1}\|_F^2+\|A_{\omega_1}-A_{\omega_0}\|_F^2)\\
		& \leq 3\left(\sup_{\omega\geq a}\Big|\frac{\partial}{\partial \omega}\|B_{\omega}\|_F^2\Big|+2\right)\varepsilon^2.
	\end{align*}
	Thus, the claimed statement follows by choosing \(M:=\sqrt{3\left(\sup_{\omega\geq a}|\frac{\partial}{\partial \omega}\|B_{\omega}\|_F^2|+2\right)}\).
	
\end{proof}
\subsection{Metric graphs with \(\delta\)-couplings}\label{sec:delta}

In this section, we indicate the modifications necessary to prove Theorem~\ref{thm:main_delta}. Using the notation in \eqref{eq:discretization}, the \(\delta\)-coupling condition \eqref{eq:delta_conditions} at $\mv \in \mV$ is equivalent to
	\[
	- \alpha_\mv \frac{x(\mv)}{\omega}
	+
	\sum_{\me \in \mE_\mv} x(\mv,\me)
	=
	0,
	\]
	leading to the perturbed matrix
	\[	
	B_\omega
	:=
	\begin{pmatrix}
	\begin{matrix}
		\dots & 0 & \dots \\
		0 & -\alpha_\mv \omega^{-1} & 0 \\
		 &  \dots &
	\end{matrix}
	&
	\rvline
	&
		\begin{matrix}
		&\dots & 0 & \dots & \\
		\dots & 1 & \dots & 1 & \dots \\
		&\dots & 0 & \dots & \\
		\end{matrix}
	\\
	\hline
		\begin{matrix}
		1 & - \cos( \ell_\me \omega)\\
		  & \sin( \ell_\me \omega)\\
		  & \dots\\
		\end{matrix}
	&
	\rvline
	&		
		\begin{matrix}
		\sin( \ell_\me \omega) & \\
		\cos( \ell_\me \omega) & 1 \\
		& \dots\\
		\end{matrix}
	\end{pmatrix}
	=
	A_\omega + \omega^{-1} P
\] 
	for \(\omega>0\) where $A_\omega$ corresponds to Laplacian where all $\delta$-couplings have been replaced by standard conditions and $P$ is an $\omega$-independent matrix. Thus, Lemma \ref{lem:pertubation} may be applied to the matrix \(B_\omega\). From here Theorem \ref{thm:main_delta} can be proved following the arguments of the proof of Theorem \ref{thm:main_tree}.

\subsection{Edgewise constant potential}
	\label{sec:potential}

We now indicate the modifications needed when an edgewise constant potential \(q=(q_\me)_{\me\in\mE}\in\mathbb R^{|\mE|}\) is added to the Laplacian, that is if one has the equation
\begin{equation}
	\label{eq:with_potential}
	- \psi'' + q_\me \psi = \lambda,
	\quad
	\text{on every edge }\me \in \mE.
\end{equation}
As we are primarily interested in the construction of a sequence of non-vanishing eigenfunctions, it is sufficient to consider large eigenvalues. More precisely, we consider positive eigenvalues $\lambda$ with \[\lambda>\max\{q_\me~|~\me\in\mE\}\]
with corresponding eigenfrequency $\omega = \sqrt \lambda$. Eigenfunctions will be a linear combination of $\sin$ and $\cos$ waves with \emph{local frequency} $\omega_\me = \sqrt{\lambda - q_\me} = \sqrt{\omega^2 - q_\me}$ on every edge $\me$. Thus, the only change to the construction of the secular function in Section~\ref{sec:secular_function} is that $\sin(\ell_\me \omega)$ and $\cos(\ell_\me \omega)$ need to be replaced by $\sin(\ell_\me \sqrt{\omega^2 - q_\me})$ and $\cos(\ell_\me \sqrt{\omega^2 - q_\me})$, respectively. We obtain the new matrix
	\[	
	B_\omega
	:=
	\begin{pmatrix}
	\begin{matrix}
		\dots & 0 & \dots \\
		0 & -\alpha_\mv \omega^{-1} & 0 \\
		 &  \dots &
	\end{matrix}
	&
	\rvline
	&
		\begin{matrix}
		&\dots & 0 & \dots & \\
		\dots & 1 & \dots & 1 & \dots \\
		&\dots & 0 & \dots & \\
		\end{matrix}
	\\
	\hline
		\begin{matrix}
		1 & - \cos( \ell_\me \sqrt{\omega^2-q_\me})\\
		  & \sin( \ell_\me \sqrt{\omega^2-q_\me})\\
		  & \dots\\
		\end{matrix}
	&
	\rvline

	&		
		\begin{matrix}
		\sin( \ell_\me \sqrt{\omega^2-q_\me}) & \\
		\cos( \ell_\me \sqrt{\omega^2-q_\me}) & 1 \\
		& \dots\\
		\end{matrix}
	\end{pmatrix}
\] 
operating on
\begin{equation}\label{eq:discretization-adjusted}
	x
	=
	\begin{pmatrix}
		\{ x(\mv) \}_{\mv \in \mV}
		\\
		\{ x(\mv,\me)\}_{\mv \in \mV, \me \in \mE_\mv}
	\end{pmatrix}
	=
	\begin{pmatrix}
		\{ \psi(v) \}_{\mv \in \mV}
		\\
		\big\{ \frac{\partial_\me \psi(\mv)}{\sqrt{\omega^2-q_\me}} \big\}_{\mv \in \mV, \me \in \mE_\mv}
	\end{pmatrix}.
\end{equation}

	\begin{lemma}\label{lem:estimate-cos-sin-pertubation}
	For all $j \in \{0,1,\dots\}$ there are $a_j, C_j  > 0$ such that for all $\omega \in [a_j, \infty)$, we have
	\begin{align*}
	\left\lvert \frac{\partial^j}{\partial \omega^j}
		\left( 
		\sin(\ell_\me \sqrt{\omega^2 - q_\me}) - \sin(\ell_\me \omega)
		\right) 
	\right\rvert
	&\leq
	\frac{C_j}{\omega}, \quad
	\text{and}\\
	\left\lvert \frac{\partial^j}{\partial \omega^j}
		\left(
		\cos(\ell_\me \sqrt{\omega^2 - q_\me}) - \cos(\ell_\me \omega) 
		\right)
	\right\rvert
	&\leq
	\frac{C_j}{\omega}.
	\end{align*}
\end{lemma}

\begin{proof}
	We only show the first identity for $j = 0$.
	The proof of the identity with cosine terms is completely analogous and the cases of higher derivatives follow inductively from similar calculations.
	Setting $a_0 := \max \sqrt{ \lvert q_\me \rvert}$ and $C_0 := \max \frac{\lvert q_\me \rvert}{\ell_\me}$, we estimate for $\omega \geq a_0$
	\begin{align*}
	|\sin(\ell_\me \sqrt{\omega^2 - q_\me}) - \sin(\ell_\me \omega)|
	& =
	\frac{1}{\ell_\me}
	\left|\int_\omega^{\sqrt{\omega^2 - q_\me}}  \cos(\ell_\me x) \mathrm{d} x\right|
	\leq
	\frac{1}{\ell_\me}	
	\lvert \sqrt{\omega^2 - q_\me} - \omega \rvert
	\\
	&=
	\frac{|q_\me|}{\ell_\me (\sqrt{\omega^2 - q_\me} + \omega)}
	\leq |q_\me|(\ell_\me \omega)^{-1}
	\leq
	\frac{C_0}{\omega}
	.
	\qedhere
	\end{align*}
\end{proof}

This allows again to express the matrix $B_\omega$ corresponding to the operator with potential in a form
\[
	B_\omega = A_\omega + \omega^{-1} P_\omega
\]
for sufficiently large \(\omega\) where again $A_\omega$ is the matrix corresponding to the free Laplacian and with all \(\delta\)-couplings replaced by standard conditions. Lemma \ref{lem:estimate-cos-sin-pertubation} shows that the pertubation $P_\omega$ indeed satisfies the conditions of Lemma \ref{lem:pertubation} and we can repeat the arguments used in the proof of Theorem \ref{thm:main_tree} to prove Theorem \ref{thm:main_potential}.
\appendix
\section{A lemma on Taylor series}
In the proofs above we use the following simple consequence of Taylor's theorem:

\begin{lemma}
	\label{lem:Taylor}
	For some \(a>0\) let $\eta \colon [a,\infty) \to \RR$ be a continuous and $(2 k + 2)$ times differentiable function with uniformly bounded $(2 k + 2)$-nd derivative and let $\alpha \neq 0$.
	Then for every $\epsilon > 0$ there is $\delta > 0$ such that the following implication holds for all $x_0 \geq a+\varepsilon$: if we have
	\[
	\max 
		\left\{ 
		\lvert \eta(x_0) \rvert, 
		\lvert \eta'(x_0) \rvert, 
		\dots, 
		\lvert \eta^{(2 k)}(x_0) \rvert,
		\lvert \eta^{(2 k + 1)}(x_0) - \alpha \rvert
		\right\}
		\leq
		\delta,
	\]
	then $\eta$ has a zero in an $\varepsilon$-neighbourhood of $x_0$.
\end{lemma}

\begin{proof}
Without loss of generality let \(\alpha>0\). Taylor's Theorem yields for $x \in [x_0 - \varepsilon, x_0 + \varepsilon]$
	\[
	\eta(x)
	=
	\sum_{l = 0}^{2 k + 1}
	\frac{(x - x_0)^l}{l!} \eta^{(l)}(x_0)
	+
	\frac{(x - x_0)^{2k+2}}{(2 k + 2)!} \eta^{(2k+2)}(\xi)
	\]
	for some $\xi \in [x_0 - \varepsilon, x_0 + \varepsilon]$.
	Choosing $\delta \leq \varepsilon^{2k+2}$ this amounts to
	\[
	\eta(x)
	=
	\alpha (x - x_0)^{2 k + 1}
	+
	\zeta (x),
	\]
	where, assuming again without loss of generality $\varepsilon \leq 1$,
	\[
	\lvert \zeta (x) \rvert
	\leq
	\delta
	\left[
	1 + \varepsilon + \frac{\varepsilon^2}{2} + \dots + \frac{\varepsilon^{2k+1}}{(2k+1)!}
	\right]
	+
	\varepsilon^{2k+2}
	\lVert \eta^{(2k+2)} \rVert_\infty
	\leq
	(\mathrm{e} + \lVert \eta^{(2k+2)} \rVert_\infty)
	\varepsilon^{2k+2}.
	\]
	Possibly shrinking $\varepsilon$, we find
	\[
	\eta(x_0-\varepsilon)\leq -\alpha \varepsilon^{2k+1} + (\mathrm{e} + \lVert \eta^{(2k+2)} \rVert_\infty)
	\varepsilon^{2k+2} 
	< 0
	\]
	and
	\[
	\eta(x_0+\varepsilon) \geq \alpha \varepsilon^{2k+1} - (\mathrm{e} + \lVert \eta^{(2k+2)} \rVert_\infty)
	\varepsilon^{2k+2} 
	> 0
	\]
	where we used \(\alpha>0\). Thus, $\eta$ changes sign in $[x_0 - \varepsilon, x_0 + \varepsilon]$ and has a zero there.
\end{proof}

\end{document}